\documentclass[12pt]{amsart}
\usepackage{fullpage}

\usepackage[OT2,OT1]{fontenc}
\newcommand\cyr{%
\renewcommand\rmdefault{wncyr}%
\renewcommand\sfdefault{wncyss}%
\renewcommand\encodingdefault{OT2}%
\normalfont
\selectfont}
\DeclareTextFontCommand{\textcyr}{\cyr}

\theoremstyle{plain}
\newtheorem{theorem}{Theorem}[section]

\newtheorem{lemma}[theorem]{Lemma}

\newtheorem{proposition}[theorem]{Proposition}

\theoremstyle{remark}
\newtheorem{remark}[theorem]{Remark}

\numberwithin{equation}{section}

\def\setof#1#2{\{#1\mid #2 \}}
\def\Z{\mathbb Z}
\def\F{\mathbb F}
\def\ism{\sim}
\def\aux{\approx}
\def\ist{\equiv}

\begin{document}

\title{Bol loops and Bruck loops of order $pq$ up to isotopism}

\author{Petr Vojt\v{e}chovsk\'y}

\address{Department of Mathematics, University of Denver, South York Street 2390, Denver, Colorado, 80208, U.S.A.}

\email{petr@math.du.edu}

\begin{abstract}
Let $p>q$ be odd primes. We classify Bol loops and Bruck loops of order $pq$ up to isotopism. When $q$ does not divide $p^2-1$, the only Bol loop (and hence the only Bruck loop) of order $pq$ is the cyclic group of order $pq$. When $q$ divides $p^2-1$, there are precisely $\lfloor(p-1+4q)(2q)^{-1}\rfloor$ Bol loops of order $pq$ up to isotopism, including a unique nonassociative Bruck loop of order $pq$.
\end{abstract}



\thanks{Research partially supported by the PROF grant of the University of Denver.}

\maketitle

\section{Introduction}

Let $p>q$ be odd primes. In this short note we classify Bol loops of order $pq$ up to isotopism, building upon the work of Niederreiter and Robinson \cite{NR1, NR2}, and Kinyon, Nagy and Vojt\v{e}chovsk\'y \cite{KNV}. The classification turns out to be a nice application of group actions on finite fields.

A \emph{quasigroup} is a groupoid $(Q,\cdot)$ in which all left translations $yL_x=xy$ and all right translations $yR_x=yx$ are bijections. A \emph{loop} is a quasigroup $Q$ with identity element $1$. A (right) \emph{Bol loop} is a loop satisfying the identity $((zx)y)x = z((xy)x)$, and a (right) \emph{Bruck loop} is a Bol loop satisfying the identity $(xy)^{-1} = x^{-1}y^{-1}$.

Two loops $Q_1$, $Q_2$ are said to be \emph{isotopic} if there are bijections $f$, $g$, $h:Q_1\to Q_2$ such that $(xf)(yg)=(xy)h$ for every $x$, $y\in Q_1$. If $f=g=h$, the loops are said to be \emph{isomorphic}. Since an isotopism corresponds to an independent renaming of rows, columns and symbols in a multiplication table, it is customary to classify loops (quasigroups and latin squares \cite{DK, LM, MW}) not only up to isomorphism but also up to isotopism.

Alongside Moufang loops \cite{Bruck, Moufang}, automorphic loops \cite{BP, KKPV} and conjugacy closed loops \cite{D, GR, K}, Bol loops and Bruck loops are among the most studied varieties of loops \cite{Bol, Glauberman1, Glauberman2, Kiechle, Nagy, Robinson}. We refer the reader to \cite{Belousov, Bruck} for an introduction to loop theory and to \cite{KNV} for an introduction to the convoluted history of the classification of Bol loops whose order is a factor of only a few primes.

\medskip

The following construction is of key importance for Bol loops of order $pq$. Let
\begin{displaymath}
    \Theta=\setof{\theta_i}{i\in \F_q}\subseteq \F_p
\end{displaymath}
be such that $\theta_0=1$ and $\theta_i^{-1}\theta_j\in\F_p^*\setminus\{-1\}$ for every $i$, $j\in \F_q$. Define $\mathcal Q(\Theta)$ on $\F_q\times \F_p$ by
\begin{equation}\label{Eq:MultLinear}
    (i,j)(k,\ell) = (i+k,\ \ell(1+\theta_k)^{-1} + (j+\ell(1+\theta_k)^{-1})\theta_i^{-1}\theta_{i+k}).
\end{equation}
Then $\mathcal Q(\Theta)$ is always a loop.

This construction was introduced and carefully analyzed by Niederreiter and Robinson in \cite{NR1}. We can restate some of their results as follows:

\begin{theorem}\label{Th:Summary}\cite{NR1}
Let $p>q$ be odd primes. Then $\mathcal Q(\Theta)$ is a Bol loop if and only if there exists a bi-infinite $q$-periodic sequence $(u_i)$ solving the recurrence relation
\begin{equation}\label{Eq:Recurrence}
    u_{n+2} = \lambda u_{n+1} - u_n
\end{equation}
for some $\lambda\in\F_p^*$ such that $u_0=1$ and $u_i^{-1}u_j\in\F_p^*\setminus\{-1\}$ for every $i$, $j$. (Then $\theta_i = u_i^{-1}$ for every $i\in\F_q$.)

If $\mathcal Q(\Theta)$ is a Bol loop then it is a Bruck loop if and only if $u_i=u_{-i}$ for every $i\in\F_q$.

Suppose that two Bol loops correspond to the sequences $(u_i)$ and $(v_i)$, respectively. Then the loops are isomorphic if and only if there is $s\in \F_q^*$ such that $u_i = v_{si}$ for every $i\in \F_q$, and the loops are isotopic if and only if there are $s\in\F_q^*$ and $r\in\F_q$ such that $u_i=v_r^{-1}v_{si+r}$ for every $i\in\F_q$.
\end{theorem}

It is not at all obvious that every Bol loop of order $pq$ is of the form $\mathcal Q(\Theta)$. This was proved in \cite{KNV}, where the isomorphism problem was resolved as follows:

\begin{theorem}\label{Th:MainKNV}\cite{KNV}
Let $p>q$ be odd primes. A nonassociative Bol loop of order $pq$ exists if and only if $q$ divides $p^2-1$. If $q$ divides $p^2-1$ then there is a unique nonassociative Bruck loop $B_{p,q}$ of order $pq$ up to isomorphism and there are precisely
\begin{displaymath}
    \frac{p-q+4}{2}
\end{displaymath}
Bol loops of order $pq$ up to isomorphism. All these loops are of the form $\mathcal Q(\setof{\theta_i}{i\in\F_q})$ with multiplication \eqref{Eq:MultLinear} and are obtained as follows:

Set $\theta_i=1$ for every $i\in\F_q$ for the cyclic group of order $pq$. For the non-cyclic loops, fix a non-square $t$ of $\F_p$, write $\F_{p^2} = \setof{u+v\sqrt{t}}{u,\,v\in\F_p}$, and let $\omega\in\F_{p^2}$ be a primitive $q$th root of unity.  Let
\begin{displaymath}
    \Gamma_{p,q} = \left\{\begin{array}{ll}
        \setof{\gamma\in\F_p}{\gamma=0\text{ or }1-\gamma^{-1}\not\in\langle\omega\rangle},&\text{if $q$ divides $p-1$},\\
        \setof{\gamma\in 1/2 + \F_p\sqrt t}{1-\gamma^{-1}\not\in\langle\omega\rangle},&\text{if $q$ divides $p+1$}.
    \end{array}\right.
\end{displaymath}
Let $f$ be the bijection on $\Gamma_{p,q}$ defined by
\begin{displaymath}
    \gamma\mapsto 1-\gamma.
\end{displaymath}
The non-cyclic Bol loops of order $pq$ up to isomorphism correspond to the orbits of the group $\langle f\rangle$ acting on $\Gamma_{p,q}$. For every orbit representative $\gamma$ let
\begin{displaymath}
    \theta_i = \theta(\gamma)_i = \frac{1}{\gamma\omega^i+(1-\gamma)\omega^{-i}}.
\end{displaymath}
The choice $\gamma=1/2$ results in the nonassociative Bruck loop $B_{p,q}$. If $q$ divides $p-1$, the choice $\gamma=1$ results in the nonabelian group of order $pq$.
\end{theorem}

Since a loop isotopic to a group is already isomorphic to it, Theorem \ref{Th:MainKNV} contains the classification of Bruck loops of order $pq$ up to isotopism. In this paper we finish the classification of Bol loops of order $pq$ up to isotopism by proving:

\begin{theorem}\label{Th:Main}
Let $p>q$ be odd primes such that $q$ divides $p^2-1$. Then there are precisely
\begin{displaymath}
    \left\lfloor \frac{p-1+4q}{2q}\right\rfloor
\end{displaymath}
Bol loops of order $pq$ up to isotopism. With the notation of Theorem $\ref{Th:MainKNV}$, these loops are obtained as follows:

Set $\theta_i=1$ for every $i\in\F_q$ for the cyclic group of order $pq$. The non-cyclic loops correspond to orbit representatives of the group $\langle f,g\rangle$ acting on $\Gamma_{p,q}$, where $g$ is given by
\begin{displaymath}
    \gamma\mapsto \frac{\gamma\omega}{\gamma\omega+(1-\gamma)\omega^{-1}}.
\end{displaymath}
\end{theorem}

\begin{remark}
Let $p>3$ be a prime. By Theorem \ref{Th:Main}, the number $N_{3p}$ of Bol loops of order $3p$ up to isotopism is equal to $\lfloor(p+11)/6\rfloor$, confirming \cite[Conjecture 7.3]{KNV}. It was shown already in \cite[p. 255]{NR1} that $N_{3p}\ge \lceil(p+5)/6\rceil$, a remarkably good estimate. Note that
\begin{displaymath}
    \left\lfloor\frac{p+11}{6}\right\rfloor - \left\lceil\frac{p+5}{6}\right\rceil = \left\{\begin{array}{ll}
        0,&\text{if $p=6k+5$},\\
        1,&\text{if $p=6k+1$}.
    \end{array}\right.
\end{displaymath}
\end{remark}

\section{Proof of the main result}

For the rest of the paper assume that $p>q$ are odd primes, $q$ divides $p^2-1$, $\omega$ is a primitive $q$th root of unity in $\F_{p^2}$ and write $\F_{p^2} =\setof{u+v\sqrt t}{u,v\in\F_p}$ for some non-square $t\in\F_p$.

Let $\mathcal X_{p,q}$ be the set of all bi-infinite $q$-periodic sequences with entries in $\F_{p^2}$. As explained in \cite{KNV}, $u\in\mathcal X_{p,q}$ solves the recurrence relation \eqref{Eq:Recurrence} if and only if $Au=\lambda u$, where $A$ is the $q\times q$ circulant matrix whose first row is equal to $(0,1,0,\dots,0,1)$ and where we identify $u$ with the vector $(u_0,\dots,u_{q-1})^T$. General theory of circulant matrices applies and yields:

\begin{lemma}\cite{KNV}\label{Lm:Circulant}
Let $A$ be the $q\times q$ circulant matrix whose first row is equal to $(0,1,0,\dots,0,1)$. For $0\le j<q$, let
\begin{equation}\label{Eq:Eigen}
    \lambda_j = \omega^j+\omega^{-j}\quad\text{and}\quad e_j = (1,\omega^j,\omega^{2j},\dots,\omega^{(q-1)j})^T.
\end{equation}
Then:
\begin{enumerate}
\item[(i)] For every $0\le j<q$, $\lambda_j$ is an element of the prime field of $\F_{p^2}$.
\item[(ii)] For every $0\le j<q$, $\lambda_j$ is an eigenvalue of $A$ over $\F_{p^2}$ with eigenvector $e_j$.
\item[(iii)] For $0< j\le (q-1)/2$, the eigenvectors $e_j$, $e_{-j}$ are linearly independent.
\item[(iv)] For $0\le j<k<q$, $\lambda_j=\lambda_k$ if and only if $j+k\equiv 0\pmod q$. In particular, $\lambda_0=2$ has multiplicity $1$, and every $\lambda_j$ with $1\le j\le (q-1)/2$ has multiplicity $2$.
\end{enumerate}
\end{lemma}

Let $\lambda_j$ and $e_j$ be as in \eqref{Eq:Eigen}. In order to better understand which elements of $\mathcal X_{p,q}$ yield Bol loops, let us define the following subsets:
\begin{align*}
    \mathcal X_{p,q}^*  &= \setof{u\in\mathcal X_{p,q}}{u_0=1},  & &\\
    \mathcal A_{p,q}^j  &= \setof{u\in\mathcal X_{p,q}^*}{Au=\lambda_j u}, &\mathcal A_{p,q}    &= \bigcup_{0\le j<q}\mathcal A_{p,q}^j,\\
    \mathcal B_{p,q}^j  &= \setof{u\in\mathcal A_{p,q}^j}{u_i^{-1}u_k\in\F_p^*\setminus\{-1\}\text{ for every }i,k},  &\mathcal B_{p,q}    &= \bigcup_{0\le j<q}\mathcal B_{p,q}^j.
\end{align*}
By Theorem \ref{Th:Summary}, the elements of $\mathcal B_{p,q}$ are precisely the sequences that yield Bol loops.

\begin{lemma}\label{Lm:ugamma}
For every $j\in\F_q$, $\mathcal A_{p,q}^j = \setof{\gamma e_j + (1-\gamma)e_{-j}}{\gamma\in\F_{p^2}}$. In particular, the only element of $\mathcal A_{p,q}^0=\mathcal B_{p,q}^0$ is the all-$1$ sequence.
\end{lemma}
\begin{proof}
Let $u\in \mathcal A_{p,q}^j$. By Lemma \ref{Lm:Circulant}, $u = \gamma e_j + \delta e_{-j}$ for some $\gamma$, $\delta\in\F_{p^2}$. The condition $u_0=1$ forces $\gamma+\delta=1$.
\end{proof}

Let $u$ be the unique element of $\mathcal B_{p,q}^0$, the all-$1$ sequence. Then $\theta_i=u_i^{-1}=1$ for every $i$, and the multiplication formula \eqref{Eq:MultLinear} becomes $(i,j)(k,\ell) = (i+k,j+\ell)$, the direct product $\Z_q\times \Z_p \cong \Z_{pq}$.

Consider the following binary relations on $\mathcal X_{p,q}^*$:
\begin{itemize}
\item $u\ism v$ if there is $s\in\F_q^*$ such that $u_i=v_{si}$ for every $i$,
\item $u\aux v$ if there is $r\in\F_q$ such that $u_i=v_r^{-1}v_{i+r}$ for every $i$, and
\item $u\ist v$ if there are $s\in\F_q^*$ and $r\in\F_q$ such that $u_i=v_r^{-1}v_{si+r}$ for every $i$.
\end{itemize}
We recognize $\ism$ as the isomorphism relation and $\ist$ as the isotopism relation from Theorem \ref{Th:Summary}.

\begin{lemma}\label{Lm:ChangeEigenvalue}
If $u\in\mathcal A_{p,q}^j$ and $v\equiv u$ via $v_i=u_r^{-1}u_{si+r}$ then $v\in\mathcal A_{p,q}^{sj}$. Conversely, if $u\in\mathcal A_{p,q}^j$ for some $j\in\F_q^*$ then for every $k\in\F_q^*$ there is $v\in\mathcal A_{p,q}^k$ such that $v\equiv u$.
\end{lemma}
\begin{proof}
Suppose that $u\in\mathcal A_{p,q}^j$ and $v_i=u_r^{-1}u_{si+r}$. Note that $v_0=u_r^{-1}u_r = 1$. By Lemmas \ref{Lm:Circulant} and \ref{Lm:ugamma}, we have $u = \gamma e_j+(1-\gamma)e_{-j}$ for some $\gamma\in\F_{p^2}$. Let $f_i = e_{j,r}^{-1}e_{j,si+r}$. Then $f_i = \omega^{-jr}\omega^{j(si+r)} = \omega^{jsi} = e_{sj,i}$. By linearity, $v = \gamma e_{sj}+(1-\gamma)e_{-sj}$. By Lemma \ref{Lm:Circulant}, $v\in\mathcal A_{p,q}^{sj}$.

For the converse, suppose that $j\in\F_q^*$ and let $s\in\F_q^*$ be such that $sj=k$. Set $v_i=u_r^{-1}u_{si+r}$ for some $r\in\F_q$. Then certainly $v\equiv u$ and we have $v\in\mathcal A_{p,q}^k$ by the first part.
\end{proof}

\begin{lemma}\label{Lm:Equivalences}
The following statements hold:
\begin{enumerate}
\item[(i)] $\ism$, $\aux$ and $\ist$ are equivalence relations on $\mathcal X_{p,q}^*$, and $\ist$ is the transitive closure of $\ism$ and $\aux$.
\item[(ii)] $\mathcal B_{p,q}$ is the union of some equivalence classes of each of $\ism$, $\aux$ and $\ist$.
\item[(iii)] If $u\in\mathcal B_{p,q}^1$ and $v_i=u_r^{-1}u_{si+r}$ then $v\in\mathcal B_{p,q}^1$ if and only if $s=\pm 1$.
\item[(iv)] $\mathcal B_{p,q}^1$ is the union of some equivalence classes of $\aux$.
\end{enumerate}
\end{lemma}
\begin{proof}
(i) Note that $\ism$ is contained in $\ist$ (set $r=0$ and use $v_0=1$) and $\aux$ is contained in $\ist$ (set $s=1$). We show that $\ist$ is an equivalence relation, the other two cases being similar. We have $u\ist u$ with $r=0$, $s=1$. If $u_i=v_r^{-1}v_{si+r}$ then $u_{-s^{-1}r}^{-1}u_{s^{-1}i-s^{-1}r} = (v_r^{-1}v_{s(-s^{-1}r)+r})^{-1}v_r^{-1}v_{s(s^{-1}i-s^{-1}r)+r} = v_i$, proving symmetry. If $u_i=v_r^{-1}v_{si+r}$ and $v_i=w_a^{-1}w_{bi+a}$ then $u_i = (w_a^{-1}w_{br+a})^{-1}w_a^{-1}w_{b(si+r)+a} = w_{br+a}^{-1}w_{(bs)i+(br+a)}$, proving transitivity. For the transitive closure, if $u_i=w_r^{-1}w_{si+r}$, set $v_i = w_r^{-1}w_{i+r}$ and note that $u_i=v_{si}$.

(ii) Suppose that $u\ist v$, $u_i=v_r^{-1}v_{si+r}$. By Lemma \ref{Lm:ChangeEigenvalue}, if $u\in\mathcal A_{p,q}$ then $v\in\mathcal A_{p,q}$. If $u_i^{-1}u_j\in\F_p^*\setminus\{-1\}$ for every $i$, $j$, then $v_{si+r}^{-1}v_{sj+r} = (v_r^{-1}v_{si+r})^{-1}v_r^{-1}v_{sj+r} = u_i^{-1}u_j\in \F_p^*\setminus\{-1\}$ for every $i$, $j$, and we are done since $(i,j)\mapsto (si+r,sj+r)$ is a bijection of $\F_q\times\F_q$.

Part (iii) follows from (ii) and Lemma \ref{Lm:ChangeEigenvalue}. Part (iv) is then immediate.
\end{proof}

Let $j\in\F_q^*$. By Lemmas \ref{Lm:ChangeEigenvalue} and \ref{Lm:Equivalences}, for any $u\in\mathcal B_{p,q}^j$ there is $v\in\mathcal B_{p,q}^1$ such that $u\ist v$, and there is no $w\in\mathcal B_{p,q}^0$ such that $u\ist w$. For the isotopism problem, it therefore remains to study the restriction of $\ist$ onto $\mathcal B_{p,q}^1$, taking parts (iii) and (iv) of Lemma \ref{Lm:Equivalences} into account.

Every element of $\mathcal B_{p,q}^1$ is by definition an element of $\mathcal A_{p,q}^1$ and hence is of the form
\begin{displaymath}
    u(\gamma) = \gamma e_1 + (1-\gamma)e_{-1}
\end{displaymath}
for some $\gamma\in\F_{p^2}$, by Lemma \ref{Lm:ugamma}. The mapping $\gamma\mapsto u(\gamma)$ is a bijection. Indeed, if $u(\gamma)=u(\delta)$ then $\gamma\omega+(1-\gamma)\omega^{-1} = u(\gamma)_1 = u(\delta)_1 = \delta\omega + (1-\delta)\omega^{-1}$, hence $(\gamma-\delta)\omega = (\gamma-\delta)\omega^{-1}$ and $\gamma=\delta$ follows. It was shown in \cite[Section 6]{KNV} that
\begin{displaymath}
    \mathcal B_{p,q}^1 = \setof{u(\gamma)}{\gamma\in\Gamma_{p,q}},
\end{displaymath}
where $\Gamma_{p,q}$ is as in Theorem \ref{Th:MainKNV}. Moreover, by \cite[Lemma 6.8]{KNV}, $\Gamma_{p,q}$ is a set of cardinality $p-q+1$, it is closed under the map $\gamma\mapsto 1-\gamma$, and it always contains $1/2$.

Let $u=u(\gamma)\in\mathcal B_{p,q}^1$ and consider $v_i=u_{-i}$. Since $u(\gamma)_i = u(1-\gamma)_{-i}$, we have $v=u(1-\gamma)\in\mathcal B_{p,q}^1$. The non-cyclic Bol loops of order $pq$ up to isomorphism therefore correspond to the orbits of the group $\langle f\rangle$ acting on $\Gamma_{p,q}$, where
\begin{displaymath}
    \gamma f  = 1-\gamma.
\end{displaymath}

At this point we can recover Theorem \ref{Th:MainKNV}. The cyclic group of order $pq$ corresponds to the unique sequence of $\mathcal B_{p,q}^0$. The above action has a unique fixed point on $\Gamma_{p,q}$, namely $\gamma=1/2$, and all other orbits have size $2$. The fixed point $\gamma=1/2$ yields a Bruck loop by Theorem \ref{Th:Summary}. Since $|\Gamma_{p,q}|=p-q+1$, there are additional $(p-q)/2$ Bol loops, for the total of $1+1+(p-q)/2 = (p-q+4)/2$ Bol loops of order $pq$. If $q$ divides $p-1$, the nonabelian group of order $pq$ must be among these $pq$ loops. It is easy to check that it is the loop corresponding to $\gamma=1$.

To further classify Bol loops of order $pq$ up to isotopism, we must now also consider the equivalence classes of $\aux$ on $\mathcal B_{p,q}^1$.

\begin{lemma}\label{Lm:gAction}
Let $\gamma$, $\delta\in\Gamma_{p,q}$. Then $u(\gamma)\aux u(\delta)$ if and only if
\begin{equation}\label{Eq:Gamma}
    \gamma = \frac{\delta\omega^r}{\delta\omega^r + (1-\delta)\omega^{-r}}
\end{equation}
for some $r\in\F_q$.
\end{lemma}
\begin{proof}
By definition, $u(\gamma)\aux u(\delta)$ if an only if there is $r\in\F_q$ such that
\begin{equation}\label{Eq:GammaDelta}
    \gamma\omega^i+(1-\gamma)\omega^{-i} = u(\gamma)_i = u(\delta)_r^{-1}u(\delta)_{i+r} = \frac{\delta\omega^{i+r}+(1-\delta)\omega^{-i-r}}{\delta\omega^r+(1-\delta)\omega^{-r}}
\end{equation}
for every $i\in\F_q$.

Suppose that \eqref{Eq:GammaDelta} holds. If $r=0$ then $u(\gamma)=u(\delta)$ and hence $\gamma=\delta$, which agrees with \eqref{Eq:Gamma}. Suppose that $r\ne 0$. Substituting $i=r$ into \eqref{Eq:GammaDelta} yields
\begin{displaymath}
    \gamma\omega^r + (1-\gamma)\omega^{-r} = \frac{\delta\omega^{2r}+(1-\delta)\omega^{-2r}}{\delta\omega^r+(1-\delta)\omega^{-r}},
\end{displaymath}
and therefore
\begin{displaymath}
    \gamma = \frac{(\delta\omega^{2r}+(1-\delta)\omega^{-2r})(\delta\omega^r+(1-\delta)\omega^{-r})^{-1}-\omega^{-r}}{\omega^r-\omega^{-r}}.
\end{displaymath}
A straightforward computation now shows that $\gamma$ is as in \eqref{Eq:Gamma}.

Conversely, suppose that $\gamma$ is as in \eqref{Eq:Gamma}. Then another straightforward calculation shows that \eqref{Eq:GammaDelta} holds for every $i$, and thus $u(\gamma)\aux u(\delta)$.
\end{proof}

For $r\in\F_q$, consider the mapping $g_r:\Gamma_{p,q}\to\Gamma_{p,q}$ defined  by
\begin{displaymath}
    \gamma g_r = \frac{\gamma\omega^r}{\gamma\omega^r + (1-\gamma)\omega^{-r}}.
\end{displaymath}
We note that $g_r$ is well-defined since $\gamma\omega^r+(1-\gamma)\omega^{-r} = u(\gamma)_r\ne 0$. By Lemma \ref{Lm:gAction}, if $\gamma = \delta g_r$ then $u(\gamma)\aux u(\delta)$, so $u(\delta)\in\mathcal B_{p,q}^1$ by Lemma \ref{Lm:Equivalences}(iv), which in turn implies $\delta\in\Gamma_{p,q}$. Altogether, $g_r$ is a bijection on $\Gamma_{p,q}$.

Yet another straightforward calculation shows that $\gamma g_rg_s = \gamma g_{r+s}$ for every $r$, $s\in\F_q$. Let $g=g_1$, that is,
\begin{displaymath}
    \gamma g = \frac{\gamma\omega}{\gamma\omega + (1-\gamma)\omega^{-1}}.
\end{displaymath}
Combining our results obtained so far, we see that $u(\gamma)\aux u(\delta)$ if and only if $\gamma$, $\delta$ are in the same orbit of the group $\langle g\rangle$ acting on $\Gamma_{p,q}$, and $u(\gamma)\ist u(\delta)$ if and only if $\gamma$, $\delta$ are in the same orbit of the group $G = \langle f,g\rangle$ acting on $\Gamma_{p,q}$.

\begin{proposition}\label{Pr:Group}
The group $G=\langle f,g\rangle$ is isomorphic to the dihedral group $D_{2q}$ of order $2q$. Moreover:
\begin{enumerate}
\item[(i)] The only fixed point of $f$ is $1/2$. If $q$ divides $p-1$ then $f(0)=1$ and $f(1)=0$.
\item[(ii)] If $0<i<q$ and $q$ divides $p-1$ then the only fixed points of $g^i$ are $0$ and $1$.
\item[(iii)] If $0<i<q$ and $q$ divides $p+1$ then $g^i$ has no fixed points.
\item[(iv)] If $0<i<q$ then the only fixed point of $fg^i$ is $(1+\omega^i)^{-1}$.
\end{enumerate}
\end{proposition}
\begin{proof}
Part (i) is obvious. For the rest of the proof, let $0<i<q$. We have $\gamma g^i=\gamma$ if and only if $\gamma\omega^i = \gamma(\gamma\omega^i+(1-\gamma)\omega^{-i})$, which is equivalent to $\gamma(1-\gamma)\omega^i = \gamma(1-\gamma)\omega^{-i}$. Clearly, $\gamma=0$, $\gamma=1$ are fixed points as long as they lie in $\Gamma_{p,q}$, which happens if and only if $q$ divides $p-1$. If $\gamma\not\in\{0,1\}$ and $\gamma g^i =\gamma$ then $\omega^i=\omega^{-i}$, a contradiction.

Suppose now that $\gamma fg^i=\gamma$. Then $(1-\gamma)g^i=\gamma$, $(1-\gamma)\omega^i = \gamma((1-\gamma)\omega^i + \gamma\omega^{-i})$, and $(1-\gamma)^2\omega^i = \gamma^2\omega^{-i}$. We certainly have $\gamma\ne 0$ and thus $((1-\gamma)/\gamma)^2=\omega^{2i}$, which we rewrite as $(1-\gamma^{-1})^2 = \omega^{2i}$. Then either $1-\gamma^{-1}=\omega^i$ (which implies $1-\gamma^{-1}\in\langle\omega\rangle$, a contradiction with $\gamma\in\Gamma_{p,q}$), or $1-\gamma^{-1}=-\omega^i$, which implies $\gamma = (1+\omega^i)^{-1}$, the only candidate for a fixed point of $fg^i$.

Now, $|f|=2$ since $f^2=1$ and $\gamma f\ne\gamma$ if $\gamma\ne 1/2$. Also $|g|=q$ since $g^q=1$ and $\gamma g\ne\gamma$ whenever $\gamma\not\in\{0,1\}$. Finally,
\begin{displaymath}
    \gamma gf = 1-\frac{\gamma\omega}{\gamma\omega+(1-\gamma)\omega^{-1}} = \frac{(1-\gamma)\omega^{-1}}{\gamma\omega+(1-\gamma)\omega^{-1}},
\end{displaymath}
while
\begin{displaymath}
    \gamma fg^{-1} = (1-\gamma)g^{-1} = \frac{(1-\gamma)\omega^{-1}}{(1-\gamma)\omega^{-1}+\gamma\omega}.
\end{displaymath}
Thus $gf=fg^{-1}$ and $G\cong D_{2q}$ follows.

Since $1/2$ is fixed by $f$ but not by $g$, the orbit-stabilizer theorem implies that the orbit of $1/2$ contains $q$ elements. In turn, each of these $q$ elements has a stabilizer of size $2$, so it must be stabilized by some $fg^i$ of $G$. We conclude that the purported fixed points $(1+\omega^i)^{-1}$ of $fg^i$ are indeed fixed points.
\end{proof}

We are ready to prove the main result, Theorem \ref{Th:Main}:

Let us count the orbits of $G=\langle f,g\rangle$ on the set $\Gamma_{p,q}$ or cardinality $p-q+1$. We will use Proposition \ref{Pr:Group} without reference. For $\gamma\in \Gamma_{p,q}$, let $O(\gamma)$ be the orbit of $\gamma$.

First suppose that $q$ divides $p-1$. Let $p-1=kq$ and note that $|\Gamma_{p,q}|=(k-1)q+2$. We have $0$, $1\in\Gamma_{p,q}$ and $O(0)=\{0,1\}$, leaving $(k-1)q$ elements. The orbit $O(1/2)$ accounts for the remaining $q$ points
fixed by some element of $G$. All the other $(k-2)q$ elements lie in orbits of size $2q$, so there must be $(k-2)/2$ such orbits. Altogether, we have counted $1+1+1+(k-2)/2 = (p-1+4q)/(2q)$ Bol loops of order $pq$ up to isotopism, including the cyclic group.

Now suppose that $q$ divides $p+1$. Let $p+1=\ell q$ and note that $|\Gamma_{p,q}| = (\ell-1)q$. Also note that $0$, $1\not\in\Gamma_{p,q}$. The orbit $O(1/2)$ again accounts for $q$ elements, and these are the only elements with nontrivial stabilizers. The remaining $(\ell-2)q$ elements lie in $(\ell-2)/2$ orbits of size $2q$. Altogether, we have counted $1+1+(\ell-2)/2 = (p+1+2q)/(2q)$ Bol loops up to isotopism. We note that $\ell$ must be even and therefore
\begin{displaymath}
    \left\lfloor\frac{p-1+4q}{2q}\right\rfloor =
    \left\lfloor\frac{p+1+4q-2}{2q}\right\rfloor =
    \left\lfloor\frac{\ell q+4q-2}{2q}\right\rfloor =
    \left\lfloor\frac{\ell}{2}+2-\frac{2}{2q}\right\rfloor =
    \frac{\ell}{2}+1 = \frac{p+1+2q}{2q},
\end{displaymath}
finishing the proof of Theorem \ref{Th:Main}.

\section*{acknowledgment}
We thank Izabella Stuhl for several discussions on the topic of this paper.

\end{document}